\documentclass[12pt]{amsart}

\usepackage{tikz}
\usetikzlibrary{calc,matrix, arrows, babel, patterns}
\usepackage{tikz-cd}

\usepackage[utf8]{inputenc}
\usepackage[T1]{fontenc}

\usepackage{amsmath,amsthm,amsfonts,amssymb,latexsym}

\usepackage{hyperref}

\usepackage{tabto}

\usepackage{cleveref}

\usepackage{color, xcolor}

\usepackage{enumerate, enumitem, verbatim, graphicx}

\usepackage{mathrsfs, thmtools}

\usepackage{setspace}

\DeclareMathOperator{\Z}{\mathbb{Z}}
\DeclareMathOperator{\Q}{\mathbb{Q}}

\DeclareMathOperator{\Aut}{Aut}

\DeclareMathOperator{\Gal}{Gal}
\DeclareMathOperator{\Irr}{Irr}
\DeclareMathOperator{\cl}{cl}

\newtheorem{theorem}{Theorem} 
\newtheorem{lemma}[theorem]{Lemma}

\newtheorem{corollary}{Corollary}[theorem]
\newtheorem{question}{Question}

\theoremstyle{remark}

\theoremstyle{definition}
\newtheorem*{definition}{Definition}

\theoremstyle{plain}

\newtheorem{Theorem}{Theorem}

\newtheorem*{conA'}{Conjecture A'}

\theoremstyle{definition}

\numberwithin{equation}{section}

\begin{document}

\title[On groups with at most five irrational conjugacy classes]{On groups with at most five irrational conjugacy classes}


\author{Gabriel A. L. Souza}
\address{Departament de Matemàtiques, Universitat de València, 46100 Burjassot, València, Spain}
\email{gabriel.area@uv.es}


\thanks{The author is supported by Ministerio de Ciencia e Innovación (grant PREP2022-000021 tied to the project PID2022-137612NB-I00 funded by MCIN/AEI/10.13039/501100011033 and ``ERDF A way of making Europe'').}



\date{\today}


\begin{abstract}
	Much work has been done to study groups with few rational conjugacy classes or few rational irreducible characters. In this paper we look at the opposite extreme. Let $G$ be a finite group. Given a conjugacy class $K$ of $G$, we say it is \textit{irrational} if there is some $\chi \in \Irr(G)$ such that $\chi(K) \not \in \Q$. One of our main results shows that, when $G$ contains at most $5$ irrational conjugacy classes, then $|\Irr_{\Q}(G)| = |\cl_{\Q}(G)|$.  This suggests some duality with the known results and open questions on groups with few rational irreducible characters. Our results are independent of the Classification of Finite Simple Groups.
\end{abstract}

\maketitle



	\section{Introduction}

Let $x$ be an element of a finite group $G$ and let $\chi$ be one of its irreducible complex characters (our notation for characters will follow that of \cite{Isaacs:CTFG}). We denote by $\Q(x)$ and $\Q(\chi)$ the field extensions of $\Q$ generated by the sets $\{\psi(x) \mid \psi \in \Irr(G)\}$ and $\{\chi(g) \mid g \in G\}$, respectively. The element $x$ is called \textbf{rational} if $\Q(x) = \Q$; otherwise, it is called \textbf{irrational}. As characters are class functions, these notions naturally give rise to those of rational and irrational conjugacy classes; the set of all rational classes is denoted $\cl_{\Q}(G)$. Likewise, the character $\chi$ is called \textbf{rational} if $\Q(\chi) = \Q$, and the set of rational irreducible characters is denoted $\Irr_{\Q}(G)$; otherwise, $\chi$ is called \textbf{irrational}.

Much work has been done to see how the two notions relate to one another. For example, writing $k(G)$ for the number of conjugacy classes of the group $G$, it is trivial that $|\cl_{\Q}(G)| = k(G)$ if and only if $|\Irr_{\Q}(G)| = k(G)$, in which case the group is said to be a \textbf{rational group}, or \textbf{$\Q$-group}. The result in the other extreme, that $|\cl_{\Q}(G)| = 1$ if and only if $|\Irr_{\Q}(G)| = 1$, however, is much less straightforward. In fact, its proof, in \cite{NavTiep:Rational}, uses the Classification of Finite Simple Groups.

In \cite{NavTiep:Rational}, the authors also prove that $|\cl_{\Q}(G)| = 2$ if and only if $|\Irr_{\Q}(G)| = 2$. They have further conjectured that the same is true if one replaces $2$ by $3$ in the previous statement. One direction of this Navarro-Tiep conjecture has been proven in \cite{Rossi:Rational}, but the other remains open as far as we are aware. 

In our work, we tackle the situation in which $G$ has few irrational characters or classes. More concretely, we have the following:

\begin{Theorem}
	\label{TheoremA}
	Let $G$ be a finite group having at most $5$ irrational conjugacy classes. Then, $|\cl_{\Q}(G)| = |\Irr_{\Q}(G)|$.
\end{Theorem}

In Section 3, we will present an example showing how one cannot replace the number $5$ by any larger integer. As a matter of fact, this example contains $4$ irrational irreducible characters and $6$ irrational conjugacy classes, showing that a version of \Cref{TheoremA} would fail for irreducible characters as opposed to conjugacy classes.

Interestingly, the Navarro-Tiep conjecture also cannot be extended further in its current statement. One of the smallest counterexamples happens to have $4$ rational conjugacy classes and $6$ rational irreducible characters, and, as we will see in Section $3$, we have not been able to find a counter-example in the other direction of that conjecture. This suggests some sort of duality between results on \emph{rational irreducible characters} and those on \emph{irrational conjugacy classes}. It would be interesting to analyze if there is any deeper reason for this apparent connection.

There has also been much work devoted to studying how the degrees of the extensions $\Q(\chi)$ and $\Q(x)$ over $\Q$ affect the structure of $G$. Let $\pi(G)$ denote the set of prime divisors of $|G|$. In \cite{Gow:Rational}, R. Gow showed that solvable rational groups satisfy $\pi(G) \subset \{2, 3, 5\}$. These results are then extended both by D. Chillag and S. Dolfi in \cite{CD:Semirational}, who considered the case where the extensions $\Q(x)$ were at most quadratic for all $x \in G$, and by J. Tent in \cite{Tent:Quadratic}, who considered the analogous situation for irreducible characters. In this direction, we show the following:

\begin{Theorem}
	\label{TheoremB}
	Let $G$ be a finite solvable group. Then, if $G$ has exactly $2$ or $3$ irrational conjugacy classes, $\pi(G) \subseteq \{2, 3, 5, 7\}$;
\end{Theorem}

We note that we were unable to find any group with order divisible by $7$ and exactly two irrational conjugacy classes, which perhaps indicates that this result could be improved for that case.

\Cref{TheoremB} suggests, comparing it to the aforementioned results of \cite{Gow:Rational}, that solvable groups with very few irrational conjugacy classes tend to resemble solvable rational groups in some ways. In fact, the number of irrational conjugacy classes always puts a bound on the prime divisors of a solvable group $G$, which is a consequence of the following result:

\begin{Theorem}
	\label{TheoremC}
	Let $n$ be a positive integer. Then, for all finite groups $G$ with exactly $n$ irrational conjugacy classes, the number of irrational irreducible characters of $G$ is bounded in terms of $n$. In particular, there exists a natural number $k = f(n)$ such that $[\Q(\chi): \Q] \leq k$ for all $\chi \in \Irr(G)$.
\end{Theorem}

The same exact result remains true when one replaces ``conjugacy classes'' by ``irreducible characters'', as will be made clear by the proof in Section 4.

\section{Preliminaries}

Before moving to the proofs of the main results, it will be useful to recall some of the basic results on rational elements and Galois actions. Throughout the rest of the paper, $G$ will always denote a finite group.

First, given $n \in \mathbb{N}$, we write $\Q_n = \Q(\zeta)$, where $\zeta$ is a primitive $n$-root of unity. It is well-known that, for all $n$, $\Gal(\Q_n {:} \Q)$ is an abelian group isomorphic to $(\Z_n)^\times$ and thus, of order $\phi(n)$, where $\phi$ is the Euler totient function. In fact, given $\sigma \in \Gal(\Q_n {:} \Q)$, there exists a unique $r_\sigma \in \Z$, with $0 < r_\sigma < n$ and $(r_\sigma, n) = 1$, such that $\sigma(\zeta) = \zeta^{r_\sigma}$. Using the ``bar convention'' for elements of $\Z_n$, the map $\sigma \mapsto \overline{r_\sigma}$ gives the explicit group isomorphism.

If $n = \exp(G)$, which is the smallest positive integer $m$ such that $x^m = 1$ for all $x \in G$, then $\Gal(\Q_n {:} \Q)$ acts on the set $\Irr(G)$ via $(\sigma \cdot \chi)(x) = \sigma(\chi(x))$. Furthermore, if $x \in G$, we may define $\sigma \cdot x = x^{r_\sigma}$. This induces an action of $\Gal(\Q_n {:} \Q)$ on $\cl(G)$, with the property that $(\sigma\cdot\chi)(x) = \chi(\sigma \cdot x)$, for all $x \in G$, $\chi \in \Irr(G)$. Using these actions, we state the following characterization of rational conjugacy classes, which will be used without further mention. We write $o(x)$ for the order of $x$ in $G$.

\begin{lemma}
	For a finite group $G$ of order $N$ and $x \in G$, the following are equivalent:
	\begin{enumerate}
		\item $x^G$ is a rational conjugacy class;
		\item $x$ is conjugate to $x^m$ for all $0 < m < n$ coprime with $n = o(x)$;
		\item $\sigma \cdot (x^G) = x^G$ for all $\sigma \in \Gal(\Q_N {:} \Q)$;
		\item $\Aut(\langle x \rangle) \cong N_G(\langle x \rangle)/C_G(x)$;
	\end{enumerate}
\end{lemma}

\begin{proof}
	Suppose there exists some $m$ coprime with $n = o(x)$ such that $x$ is not conjugate to $x^m$, with $0 < m < n$. Then, writing $|G| = uv$, where the set of prime divisors of $u$ and $n$ are equal, and $(n, v) = 1$, we can choose $0 < M < N$ such that $M \equiv m \pmod{n}$ and $M \equiv 1 \pmod{v}$. Thus, $M$ is coprime to $N$ and $x^M = x^m$. Taking $\sigma \in \Gal(\Q_N {:} \Q)$ such that $\sigma(\zeta) = \zeta^M$, there must be some irreducible character $\chi \in \Irr(G)$ such that $\chi(x) \neq \chi(x^m) = \chi(\sigma \cdot x)$. Hence, $\chi(x) \neq \sigma(\chi(x))$, meaning $\chi(x) \not \in \Q$ and $x^G$ is not rational. 
	
	On the other hand, if $x$ is conjugate to $x^m$ for all $m$ coprime with $n$, then the same logic as before shows that $x$ is conjugate to $x^k$ for all $k$ coprime with $|G|$. This means by definition, $\sigma \cdot x^G = x^G$ for all $\sigma \in \Gal(\Q_N {:} \Q)$. This last condition also implies $\sigma(\chi(x)) = \chi(\sigma \cdot x) = \chi(x)$ for all $\sigma \in \Gal(\Q_N {:} \Q)$ and for all $\chi \in \Irr(G)$, meaning $x^G$ is rational.
	
	To see the equivalence between the first three statements and the fourth one, note that $N_G(\langle x \rangle)$ acts on $\langle x \rangle$ by conjugation, with kernel $C_G(x)$. This means we may identify the quotient $N_G(\langle x \rangle)/C_G(x)$ as a subgroup of $\Aut(\langle x \rangle)$. Also, the latter consists of maps sending $x$ to $x^k$ for all $0 < k < n$ coprime with $n$. Thus, it follows from the reasoning before that $x^G$ is rational if and only if every element of $\Aut(\langle x \rangle)$ can be obtained by the conjugation action of $N_G(\langle x \rangle)$.
\end{proof}

We also collect the following basic results on rational and irrational elements, which can be helpful in certain situations.

\begin{lemma}
	\label{basicLemma}
	Let $G$ be a finite group. Then:
	\begin{enumerate}[label=(\alph*)]
		\item If $H \leq G$ and $x \in H$ is rational in $H$, then $x$ is rational in $G$;
		\item If $N \unlhd G$ and $x \in G$ is rational, then $xN$ is rational in $G/N$;
		\item If $N \unlhd G$, $(o(x), |N|) = 1$ and $x$ is irrational in $G$, then $xN$ is irrational in $G/N$;
		\item If $x \in G$ is rational, then so are all powers of $x$;
		\item If $x \in G$ is irrational and $z \in C_G(x)$ has order coprime with that of $x$, then $xz$ is also irrational;
	\end{enumerate}
\end{lemma}

\begin{proof}
	Items (a) and (b) follow from the previous lemma. Item (c) is an immediate consequence of \cite[Lemma 5.1 (e)]{NavTiep:Rational} and item (d) is \cite[Lemma 5.1 (d)]{NavTiep:Rational}. 
	
	Finally, for (e), take some $m$ coprime with $o(x)$ such that $x$ is not conjugate to $x^m$. Take $k \in \mathbb{N}$ such that $k \equiv m \pmod{o(x)}$ and $k \equiv 1 \pmod{o(z)}$. Then, $k$ is coprime to $o(xz) = o(x)o(z)$ and, assuming $xz$ is rational, we would need to have $xz$ conjugate to $(xz)^k = x^m z$. But then, elevating to the power $o(z)$, $x^{o(z)}$ is conjugate to $x^{mo(z)}$. As $o(z)$ is coprime to $o(x)$, we can pick $j \in \mathbb{N}$ such that $o(z) j \equiv 1 \pmod{o(x)}$. Then, $x = x^{o(z)j}$ is conjugate to $x^{m}$, a contradiction.
\end{proof}

We make heavy use of the following standard result, which is a consequence of Brauer's Lemma on Character Tables. Recall that two actions of a group $G$ on sets $\Omega_1$ and $\Omega_2$ are deemed \emph{permutation isomorphic} if there exists a bijection $f: \Omega_1 \to \Omega_2$ satisfying $f(g\cdot \omega) = g \cdot f(\omega)$ for all $\omega \in \Omega_1$.

\begin{lemma}
	\label{BLCT}
	Let $G$ be a finite group, let $n = \exp(G)$ and let $\mathcal{G} = \Gal(\Q_n {:} \Q)$. Then, if $\mathcal{H}$ is any cyclic subgroup of $\mathcal{G}$, the actions of $\mathcal{H}$ on $\Irr(G)$ and $\cl(G)$ are permutation isomorphic.
\end{lemma}

\begin{proof}
	See \cite[Theorem 3.3]{Navarro:McKay}.
\end{proof}

In order to prove some of our main theorems, we will also need a fact about abelian subgroups of $S_n$, the symmetric group on $n$ symbols.

\begin{lemma}
	\label{SubgroupsSn}
	Let $G$ be an abelian subgroup of $S_n$ and write $G \cong C_{q_1} \times \cdots \times C_{q_k}$, where $1 \neq q_i$ for all $i$ and they are all prime powers. Then, $q_1 + \cdots + q_k \leq n$.
\end{lemma}
\begin{proof}
	Consider the orbits $\{\mathcal{O}_1, ..., \mathcal{O}_t\}$ of the natural action of $G$ on the set $\{1, ..., n\}$. Then, there exists an action of $G$ on each orbit $\mathcal{O}_i$. This can be expressed as a homomorphism $f_i : G \to S_{l_i}$, where $l_i = |\mathcal{O}_i|$. The action of $G$ is equivalent to that of $G/\ker f_i$, and so, denoting $f_i(G) = G_i$, $G_i$ is an abelian subgroup of $S_{l_i}$ acting transitively on $\mathcal{O}_i$. 
	
	Fix some $x \in \mathcal{O}_i$. For all $y \in \mathcal{O}_i$, there exists some $\sigma_y \in G_i$ such that $\sigma_y(x) = y$, by transitivity. Furthermore, $\sigma_y$ is uniquely determined. Indeed, if $\tau$ has the same property, then, given $z \in \mathcal{O}_i$, take $\psi \in G_i$ with $\psi(x) = z$. Then, $\sigma_y(z) = \sigma_y (\psi(x)) = \psi(\sigma_y(x)) = \psi(\tau(x)) = \tau(\psi(x)) = \tau(z)$. As $z$ is arbitrary, this shows $\sigma_y = \tau$. Ergo, $\sigma_y$ is uniquely determined by $y$, and one easily sees that $y \mapsto \sigma_y$ is a bijection from $\mathcal{O}_i$ to $G_i$. In particular, $|G_1| + \cdots + |G_t| = |\mathcal{O}_1| + \cdots + |\mathcal{O}_t| = n$. 
	
	Now, $G$ can naturally be embedded in the product $G_1 \times \cdots \times G_t$ via $f: G \to G_1 \times \cdots \times G_t$ mapping $g$ to $(f_1(g), ..., f_t(g))$; indeed, this is injective because, if $g \in \bigcap_i \ker f_i$, then $g$ acts trivially on all points of all orbits, meaning $g = 1$. In particular, $C_{q_i} \leq G_1 \times \cdots \times G_t$ for all $i$. Since the order of an element of a direct product is the least common multiple of the orders of the factors and the $q_i$ are prime powers, it follows, for each $i$, that there exists $j \in \{1, ..., t\}$ such that $C_{q_i} \subset G_j$. Since the sum of two positive integers greater then $1$ is always less than or equal to their product, one sees that $q_1 + \cdots + q_k \leq |G_1| + \cdots + |G_t| = n$.
\end{proof}

\begin{corollary}
	\label{keyCorollary}
	An abelian subgroup of $S_n$ can be expressed as a direct product of at most $n/2$ non-trivial groups.
\end{corollary}

\begin{proof}
	This is a consequence of the previous result and the primary decomposition of finite abelian groups.
\end{proof}

\section{Proof of Theorem A}

With the previous elementary facts dealt with, we can move on to the proof of our first main theorem. In order to do so, we need a simple lemma. Notice how each rational conjugacy class is a fixed point of the action by the appropriate Galois group, meaning there is an induced action on the irrational conjugacy classes.

\begin{lemma}
	\label{mainLemma}
	Let $G$ be a group with exactly $4$ irrational conjugacy classes, let $n = \exp(G)$ be its exponent and let $\varphi: \Gal(\Q_n{:}\Q) \to S_4$ be the homomorphism corresponding to the Galois action on the four irrational classes. Then, unless possibly when the image of $\varphi$ is the Klein $4$-group $K \unlhd S_4$, $G$ has exactly $4$ irrational irreducible characters.
\end{lemma} 

\begin{proof}
	Since $\Gal(\Q_n{:}\Q)$ is abelian, the image of the action is an abelian subgroup of $S_4$. This narrows the possibilities to either a cyclic group, or one of the two copies of $C_2 \times C_2$ up to conjugation: one that contains a $2$-cycle and the other that does not. In the first case, the result follows from \Cref{BLCT}. Thus, without loss of generality, all that is left is to deal with the case in which the image of $\varphi$ is $\{1, (12), (34), (12)(34)\}$.
	
	Let $\sigma$ be in the preimage of $(12)$ and $\tau$, of $(34)$. Then, $\sigma \tau$ acts as $(12)(34)$. \Cref{BLCT} gives us exactly $4$ irreducible characters moved by $\sigma \tau$, none of which can be rational; call them $\chi_1, \chi_2, \chi_3, \chi_4$.
	
	Suppose there is some fifth irrational character $\chi_5$ not on that list. Then, some element of $\Gal(\Q_n{:}\Q)$ acts non-trivially upon it. Suppose, without loss of generality, that it is $\sigma$. If $\sigma(\chi_5) = \chi_6$ for some sixth irrational character $\chi_6$, then $\tau$ would permute $\{\chi_i \mid 1 \leq i \leq 6\}$ without fixed points, by construction, contradicting \Cref{BLCT}.
	
	Then, we can assume without loss of generality $\sigma(\chi_5) = \chi_1$. Since $(\tau \sigma)(\chi_1) = \chi_2$, this forces $\tau(\chi_5) = \chi_2$. \Cref{BLCT} applies again with the subgroup induced from $\sigma$, meaning it can only act as a transposition on these characters (i.e., $\sigma(\chi_2) = \chi_2$).  Thus, $(\sigma \tau)(\chi_5) = \chi_2$ and $(\tau \sigma)(\chi_5) = \chi_1$, a contradiction.
\end{proof}

We may now go on to prove the main result. To that end, note that the cases where $G$ has exactly two or exactly three irrational conjugacy classes follow immediately from \Cref{BLCT}, since $\Gal(\Q_n{:}\Q)$ acts cyclically in both cases. 

Furthermore, in the case where $G$ has exactly five irrational conjugacy classes, the Galois group acts as an abelian subgroup of $S_5$ which cannot be identified with a subgroup of $S_4$. By this, we mean that, for all $i$ in $\{1, 2, 3, 4, 5\}$, there must be some permutation in the Galois group that moves $i$, since all five correspond to irrational classes. A simple inspection (for instance, in the \texttt{GAP} computer system, \cite{GAP4}) shows that the only abelian subgroups of $S_5$ with this property are cyclic, meaning \Cref{BLCT} applies again.

With the previous observations in mind, we follow the proof below with the assumption that $G$ has exactly four irrational conjugacy classes. We write $x \sim y$ when the elements $x, y$ are conjugate in $G$, and begin with two computational lemmas which will be useful later.

\begin{lemma}
	\label{ColumnAnalysis}
	Let $G$ be a finite group and let $\sigma \in \Gal(\Q_n{:}\Q)$, where $n = \exp(G)$, be a Galois automorphism. Suppose $\sigma$ leaves all but $4$ irreducible characters fixed and has two orbits of size $2$ in $\Irr(G)$. Then, if $X = \{\chi_1, \sigma \chi_1, \chi_2, \sigma \chi_2\}$ are the characters not fixed by $\sigma$ and $x \in G$ is a representative of a conjugacy class not fixed by $\sigma$, we have
	\begin{equation*}
		|C_G(x)| = |\chi_1(x) - \sigma\chi_1(x)|^2 + |\chi_2(x) - \sigma\chi_2(x)|^2.
	\end{equation*}
\end{lemma}

\begin{proof}
	Write $x^a = \sigma \cdot x$, with $a > 1$ and $x^a \not \sim x$. By the second orthogonality relation, we have
	\begin{equation}
		\label{eq1}
		\sum_{\chi \in X} \chi(x)\overline{\chi(x^a)} + \sum_{\chi \not \in X} \chi(x)\overline{\chi(x^a)} = \sum_{\chi \in \Irr(G)} \chi(x)\overline{\chi(x^a)} = 0.
	\end{equation}
	
	Using the second orthogonality relation again, we get:
	\begin{equation*}
		\sum_{\chi \in X} \chi(x)\overline{\chi(x)} + \sum_{\chi \not \in X} \chi(x)\overline{\chi(x)} = \sum_{\chi \in \Irr(G)} \chi(x)\overline{\chi(x)} = |C_G(x)|,
	\end{equation*}
	meaning, then, that
	\begin{equation}
		\label{eq2}
		\sum_{\chi \not \in X} \chi(x)\overline{\chi(x)} = |C_G(x)| - \sum_{\chi \in X} \chi(x)\overline{\chi(x)}.
	\end{equation}
	Also, if $\chi \not \in X$, then it is fixed by $\sigma$; equivalently, $\chi(x^a) = \chi(\sigma \cdot x) = \sigma(\chi(x)) = \chi(x)$ in this case. This allows us to substitute \Cref{eq2} into \Cref{eq1}, yielding
	\begin{equation*}
		\sum_{\chi \in X} \chi(x)\overline{\chi(x^a)} + |C_G(x)| - \sum_{\chi \in X} \chi(x)\overline{\chi(x)} = 0,
	\end{equation*}
	which is equivalent to 
	\begin{equation*}
		\sum_{\chi \in X} \chi(x)\overline{(\chi(x) - \sigma\chi(x))} = |C_G(x)|.
	\end{equation*}
	Furthermore
	\begin{align*}
		\sum_{\chi \in X} \chi(x)\overline{(\chi(x) - \sigma\chi(x))} 
		= |\chi_1(x) - \sigma\chi_1(x)|^2 + |\chi_2(x) - \sigma\chi_2(x)|^2,
	\end{align*}
	by expanding the sum and using the fact that $\sigma^2$ acts as the identity. Putting the last two equations together, we obtain
	\begin{equation*}
		|C_G(x)| = |\chi_1(x) - \sigma\chi_1(x)|^2 + |\chi_2(x) - \sigma\chi_2(x)|^2,
	\end{equation*}
	which yields the result.
\end{proof}

\begin{lemma}
	\label{RowAnalysis}
	Let $G$ be a finite group, let $n = \exp(G)$ and suppose $\Gal(\Q_n{:}\Q)$ acts on $\cl(G)$ as the Klein group $K \unlhd S_4$. Then, if $\mathscr{D} = \{x, x^a, x^b, x^{ab}\}$ is a complete set of representatives of the irrational conjugacy classes moved by $K$, $\sigma \in \Gal(\Q_n{:}\Q)$ is such that $\sigma \cdot x = x^a$ and $\chi \in \Irr(G)$ is such that $\chi \neq \sigma\chi$, we have
	\begin{equation*}
		|C_G(x)| = |\chi(x) - \chi(x^a)|^2 + |\chi(x^b) - \chi(x^{ab})|^2.
	\end{equation*}
\end{lemma}

\begin{proof}
	Let $\mathscr{C}$ be a complete set of representatives of the conjugacy classes of $G$ containing the subset $\mathscr{D}$. By the first orthogonality relation, since $\sigma\chi \neq \chi$, we get
	\begin{equation}
		\label{eq4}
		\sum_{y \in \mathscr{D}} |y^G| \chi(y)\overline{\sigma\chi(y)} + \sum_{y \in \mathscr{C} \setminus \mathscr{D}} |y^G| \chi(y)\overline{\sigma\chi(y)} =  \sum_{y \in \mathscr{C}} |y^G|\chi(y)\overline{\sigma\chi(y)} = 0.
	\end{equation}
	
	We may also use the first orthogonality relation again, to get
	\begin{equation*}
		\sum_{y \in \mathscr{D}} |y^G| \chi(y)\overline{\chi(y)} + \sum_{y \in \mathscr{C} \setminus \mathscr{D}} |y^G| \chi(y)\overline{\chi(y)} = \sum_{y \in \mathscr{C}} |y^G| \chi(y)\overline{\chi(y)} = |G|,
	\end{equation*}
	meaning that
	\begin{equation}
		\label{eq5}
		\sum_{y \in \mathscr{C} \setminus \mathscr{D}} |y^G| \chi(y)\overline{\chi(y)} = |G| - \sum_{y \in \mathscr{D}} |y^G| \chi(y)\overline{\chi(y)}.
	\end{equation}
	
	If $y \in \mathscr{C} \setminus \mathscr{D}$, its conjugacy class is fixed by $K$, meaning $\sigma\chi(y) = \chi(\sigma \cdot y) = \chi(y)$. We may thus replace \Cref{eq5} on \Cref{eq4}, yielding
	\begin{equation*}
		\sum_{y \in \mathscr{D}} |y^G| \chi(y)\overline{\sigma\chi(y)} + |G| - \sum_{y \in \mathscr{D}} |y^G| \chi(y)\overline{\chi(y)} = 0,
	\end{equation*}
	or, equivalently,
	\begin{equation*}
		\sum_{y \in \mathscr{D}} |y^G| \chi(y)\overline{(\chi(y) - \sigma \chi(y))} = |G|.
	\end{equation*}
	Also, all members of $\mathscr{D}$ have centralizers of the same size, as they are generators of the same cyclic group. In particular, $|y^G| = |x^G|$ for all $y \in \mathscr{D}$, meaning the above equation can be rewritten as
	\begin{equation*}
		\sum_{y \in \mathscr{D}} \chi(y)\overline{(\chi(y) - \sigma \chi(y))} = |C_G(x)|.
	\end{equation*}
	We can open the sum explicitly, in order to obtain
	\begin{align*}
		\sum_{y \in \mathscr{D}} \chi(y)\overline{(\chi(y) - \sigma \chi(y))} 
		= |\chi(x) - \chi(x^a)|^2 + |\chi(x^b) - \chi(x^{ab})|^2.
	\end{align*}
	Therefore, we get
	\begin{equation*}
		|C_G(x)| = |\chi(x) - \chi(x^a)|^2 + |\chi(x^b) - \chi(x^{ab})|^2
	\end{equation*}
	as desired.
\end{proof}

\begin{proof}[Proof of \Cref{TheoremA}]
	Write $|G| = n$. By the Lemma, we may assume the Galois group $\Gal(\Q_n{:}\Q)$ acts as $\{1, (12)(34), (13)(24), (14)(23)\}$, and we can take $\sigma$ and $\tau$ that act as $(12)(34)$ and $(13)(24)$, respectively. Hence, if $x$ is a representative of an irrational conjugacy class of $G$, the other classes can be written as $\sigma \cdot x$, $\tau \cdot x$ and $\sigma \tau \cdot x$. Equivalently, if $o(x) = m$, there are natural numbers $a, b$ such that $(m, a) = 1 = (m, b)$ and the four irrational classes are represented by $x, x^a, x^b, x^{ab}$, with $x^{a^2} \sim x$ and $x^{b^2} \sim x$, where $\sigma \cdot x = x^a$ and $\tau \cdot x = x^b$.
	
	By \Cref{ColumnAnalysis}, we know
	\begin{equation*}
		|C_G(x)| = |\chi_1(x) - \sigma\chi_1(x)|^2 + |\chi_2(x) - \sigma\chi_2(x)|^2.
	\end{equation*}
	Also, by \Cref{RowAnalysis}, we know
	\begin{equation*}
		|C_G(x)| = |\chi_1(x) - \chi_1(x^a)|^2 + |\chi_1(x^b) - \chi_1(x^{ab})|^2.
	\end{equation*}
	Using that $\chi_1(x^a) = \sigma\chi(x)$, and the analogous equality for $\tau$, we may combine both equations to obtain
	\begin{equation}
		\label{main}
		|\tau\chi_1(x) - \sigma\tau\chi_1(x)| = |\chi_2(x) - \sigma\chi_2(x)|.
	\end{equation}
	
	Suppose, now, that $\tau$ fixes $\chi_1$. If that is the case, using \Cref{BLCT}, there must be two orbits of size two induced by $\tau$ on the characters of $G$. If $\chi_2$ wasn't in any such orbit, $\sigma\tau$ would act non-trivially on more than four characters, which is impossible, by the same lemma. Thus $\tau \chi_2 \neq \chi_2$. 
	
	If $\sigma \tau \chi_2 \neq \chi_2$, then the characters $\{\chi_2, \sigma \chi_2, \tau\chi_2, \sigma \tau \chi_2\}$ are all distinct and account for two orbits of size $2$ for $\sigma$. Then, either $\chi_1 = \tau \chi_2$ or $\chi_1 = \sigma \tau \chi_2$, neither of which is fixed by $\tau$, contradicting our assumption.
	
	So, with the assumption that $\tau$ fixes $\chi_1$, we have $\tau \chi_2 = \sigma \chi_2$ and there is an irrational irreducible character $\chi_3$ such that $Y = \{\chi_2, \tau\chi_2, \chi_3, \tau \chi_3\}$ accounts for the two orbits of size $2$ of $\tau$ in the (irrational) irreducible characters of $G$. 
	
	By \Cref{ColumnAnalysis}, using $\tau$ instead of $\sigma$, we obtain
	\begin{equation*}
		|C_G(x)| = |\chi_2(x) - \tau\chi_2(x)|^2 + |\chi_3(x) - \tau\chi_3(x)|^2.
	\end{equation*}
	Also, by \Cref{RowAnalysis}, using $\chi_2$ and $\tau$ instead of $\chi_1$ and $\sigma$, we get
	\begin{equation*}
		|C_G(x)| = |\chi_2(x) - \chi_2(x^b)|^2 + |\chi_2(x^a) - \chi_2(x^{ab})|^2.
	\end{equation*}
	
	We now use $\tau \chi_2(x) = \chi_2(x^b)$ to put them together and obtain
	\begin{equation*}
		|\chi_2(x) - \sigma\chi_2(x)| = |\chi_3(x) - \tau\chi_3(x)|,
	\end{equation*}
	by using that $\sigma \chi_2 = \tau \chi_2$. Combining this with \Cref{main}, we get the simultaneous equalities
	\begin{equation}
		\label{eqMain}
		|\chi_1(x) - \sigma\chi_1(x)| = |\chi_2(x) - \sigma\chi_2(x)| = |\chi_3(x) - \tau\chi_3(x)|.
	\end{equation}
	
	Now, notice that both $\alpha = \chi_1(x) + \sigma\chi_1(x)$ and $\beta = \chi_1(x)\sigma\chi_1(x)$ are rational numbers, since they are fixed by all Galois automorphisms. Thus, $\chi_1(x)$ and $\sigma\chi_1(x)$ are the roots of the polynomial $t^2 - \alpha t + \beta$, meaning $|\chi_1(x) - \sigma\chi_1(x)| = |\sqrt{\Delta}|$, where $\Delta$ is the (non-zero) discriminant of the polynomial. Analogously (as $\sigma\tau$ fixes $\chi_2$ and $\sigma$ fixes $\chi_3$), we get some non-zero rational numbers $\Delta', \Delta''$ such that $|\chi_2(x) - \sigma\chi_2(x)| = |\sqrt{\Delta'}|$ and $|\chi_3(x) - \tau\chi_3(x)| = |\sqrt{\Delta''}|$.
	
	By the Pigeonhole Principle, at least two such numbers, say, $\Delta$ and $\Delta'$, have the same sign (i.e., they are either both positive or both negative). In that case, since by \Cref{eqMain}, they both have the same norm, it is straightforward that $\sqrt{\Delta} = \sqrt{\Delta'}$. However, $\chi_2(x) = \gamma + \delta \sqrt{\Delta'}$ for some $\gamma, \delta \in \Q$. Thus, $\chi_2(x) = \gamma + \delta \sqrt{\Delta}$ and then $\tau\chi_2(x) = \gamma + \delta \tau(\sqrt{\Delta}) = \chi_2(x)$, as $\tau$ fixes $\chi_1(x)$ and, therefore, it also fixes $\sqrt{\Delta}$. This would mean $\sigma\chi_2(x) = \chi_2(x)$, implying $\chi_2(x)$ is rational, which is contrary to hypothesis.
	
	Ergo, we cannot have $\tau \chi_1 = \chi_1$. The exact same logic with $\sigma \tau$ instead of $\tau$ implies we also cannot have $\sigma \tau \chi_1 = \chi_1$. Then, the characters $\{\chi_1, \sigma \chi_1, \tau\chi_1, \sigma \tau \chi_1\}$ are all distinct and account for two orbits of size $2$ for each of the Galois automorphisms. By \Cref{BLCT}, this finishes the proof.
\end{proof}

The analogous result for irrational characters instead of classes is true when $G$ has exactly two, three or five irrational characters, since these three cases amount to \Cref{BLCT}, as mentioned before. However, when there are exactly four irrational characters, the result is false in general. One counter-example is ``\texttt{SmallGroup(32, 42)}'' in the \texttt{GAP} ``\texttt{SmallGroups}'' library. It has $4$ irrational characters, but $6$ irrational conjugacy classes, and thus it also shows that an extension of \Cref{TheoremA} for six or more irrational conjugacy classes would be false.

This is analogous, as mentioned in the Introduction, to how an extension of the Navarro-Tiep conjecture for four rational classes would also fail to hold, as shown by ``\texttt{SmallGroup(32, 15)}'', which contains exactly $4$ rational conjugacy classes but $6$ rational characters. We cannot help but notice how these were, respectively, the number of \emph{irrational characters} and \emph{irrational conjugacy classes} in \texttt{SmallGroup(32, 42)}, illustrating the apparent duality mentioned in the Introduction. 

So, we now know that having up to $5$ irrational classes implies $|\cl_{\Q}(G)| = |\Irr_{\Q}(G)|$, and that the same result is false when we consider either ``irrational characters'' or ``rational classes'' instead. If the apparent duality were to have some reason for being, one would expect the following question to have an affirmative answer:

\begin{question}
	\label{Navarro-Tiep-extension}
	Let $G$ be a group with at most $5$ rational irreducible characters. Then, is it true that $|\cl_{\Q}(G)| = |\Irr_{\Q}(G)|$?
\end{question}

Checking all groups in the \texttt{GAP} ``\texttt{SmallGroups}'' library of order up to $1000$, we have come short of a counterexample. In \cite{Rossi:Rational}, the author cites the group \texttt{SmallGroup(672, 128)} as having $4$ rational irreducible characters and $6$ rational conjugacy classes. However, upon examination, we find that the numbers of rational characters and conjugacy classes were unfortunately switched around; if $G = \texttt{SmallGroup(672, 128)}$, then $|\Irr_{\Q}(G)| = 6$ and $|\cl_{\Q}(G)| = 4$, invalidating this group as a counterexample. At least in the restricted case of $2$-groups, there can be none, as we now show.

\begin{theorem}
	\Cref{Navarro-Tiep-extension} is true for $2$-groups.
\end{theorem}

\begin{proof}
	First of all, the main result in \cite{NavTiep:Rational} proves the case for $2$ rational irreducible characters for all groups, as was mentioned in the Introduction. Also, notice that a $2$-group cannot have exactly $3$ rational irreducible characters. Indeed, a cyclic $2$-group has exactly $2$ rational irreducible characters and, for noncyclic groups, $[G:\Phi(G)] \geq 4$, which in itself induces $4$ rational irreducible characters. 
	
	For the case in which $G$ has exactly $4$ rational irreducible characters, \cite[Theorem 4.4]{IsaacsNavarro:p-groups} shows $G = XY$, with $X, Y$ cyclic, $X \unlhd G$ and $|X \cap Z(G)| \geq 4$. Write $X = \langle a \rangle$ and $Y = \langle b \rangle$. Then, we may take some element $a^l$ of order $4$ in $X \cap Z(G)$. Since $(a^l)^b = a^l$, we get $kl \equiv l \pmod{o(a)}$, meaning $k \equiv 1 \pmod{4}$, as $o(a^l) = 4$. By \cite[Theorem 6]{TentSangroniz:2-groups}, these groups have exactly $4$ rational conjugacy classes.
	
	Finally, if $G$ has exactly $5$ rational irreducible characters, $G$ is of maximal class by one of the main theorems of \cite{IsaacsNavarro:p-groups}. 
	
	If $G$ is dihedral of order $2^{n+1}$, write $G = \langle x, y \mid x^2 = y^{2^n} = 1, yx = xy^{-1} \rangle$. The conjugacy classes of $G$ are known to be 
	\begin{equation*}
		\{1\}, \{y^{2^{n-1}}\}, \{x, xy^2, ..., xy^{2^n-2}\}, \{xy, xy^3, ..., xy^{2^n-1}\} \text{ and } \{y^i, y^{-i}\}, \text{ for } 1 \leq i < 2^{n-1}.
	\end{equation*} 
	From this, it is clear than the first four are rational, as is $\{y^{2^{n-2}}, y^{-2^{n-2}}\}$, and that all other classes are irrational, by order considerations.
	
	If $G$ is semidihedral of order $2^{n+1}$, write $G = \langle x, y \mid x^2 = y^{2^n} = 1, yx = xy^{2^{n-1}-1} \rangle$. Its conjugacy classes are known to be 
	\begin{equation*}
		\{1\}, \{y^{2^{n-1}}\}, \{x, xy^2, ..., xy^{2^n-2}\}, \{xy, xy^3, ..., xy^{2^n-1}\} \text{ and } \{y^i, y^{i(2^{n-1}-1)}\},
	\end{equation*} 
	where, $i \neq 2^{n-1}$ (of course, some classes are repeated in this last list). In particular, the first four are rational and so is $\{y^{2^{n-2}}, y^{-2^{n-2}}\}$; by order considerations, those are all of them.
	
	Finally, if $G$ is generalized quaternion of order $2^{n+1}$, write $G = \langle x, y \mid x^2 = y^{2^{n-1}}, y^{2^n} = 1, yx = xy^{-1} \rangle$. Just like in the dihedral case, the conjugacy classes of $G$ are known to be
	\begin{equation*}
		\{1\}, \{y^{2^{n-1}}\}, \{x, xy^2, ..., xy^{2^n-2}\}, \{xy, xy^3, ..., xy^{2^n-1}\} \text{ and } \{y^i, y^{-i}\}, \text{ for } 1 \leq i < 2^{n-1}.
	\end{equation*} 
	Looking at this list, we once again get exactly five rational classes, which finishes the proof.
\end{proof}

\section{Bounding prime divisors}

In this section, we are going to work towards a proof of Theorems \ref{TheoremB} and \ref{TheoremC}. In order to do so, we will be using much material from \cite{FariasESoares:Primes}, so it is useful to recall some of it here. First, a fairly technical definition which will be frequently used.

\begin{definition}
	Suppose a group $G$ acts on a finite-dimensional vector space $V$ over a finite field $F$. The action is said to have the \textbf{$k$-eigenvalue property} for some $k \in \mathbb{N}$, if $k$ divides $|F^\times|$ and for every $v \in V$, there exists some $g \in G$ such that $g \cdot v = \lambda v$, where $\lambda$ is a fixed element of order $k$ in $F^\times$.
\end{definition}

The most important application of this result for what is to come is \cite[Theorem B]{FariasESoares:Primes}, which we partially restate here for convenience.

\begin{theorem}[Farias e Soares]
	\label{FariasESoares}
	Let $G$ be a solvable group acting on a finite-dimensional vector space $V$ over $\mathbb{F}_p$, where $p \nmid |G|$. Then, if the action has the $k$-eigenvalue property and $\varphi$ is the Brauer character of $G$ afforded by $V$, then either $[\Q(\varphi):\Q] \geq p/\sqrt{3}$ or $\Q(\varphi)$ contains a primitive $(k/(k, 4))$-root of unity. In any case, if $q$ is a prime divisor of $k$, either $\Q(\varphi)$ contains a primitive $q$-root of unity or $q = 3$ and $p \leq 7$.
\end{theorem}

As $p$ does not divide the order of $G$, the Brauer character in the previous theorem is always a complex character, a fact which will be important later. 

\begin{proof}[Proof of \Cref{TheoremB}]
	We proceed by induction. Where convenient, we will refer to the case where $G$ contains exactly two irrational conjugacy classes as ``Case (a)'' and to that in which there are exactly three irrational conjugacy classes as ``Case (b)''. 
	
	Suppose $G$ contains more than one minimal normal subgroup and let $M_1 \neq M_2$ be two such subgroups. Then, $G$ is naturally embedded in the direct product $G/M_1 \times G/M_2$, and each factor also has at most two (resp. three, in case (b)) irrational classes, as can easily be seen. Thus, the induction hypothesis applied to each one shows $\pi(G) \subset \{2, 3, 5, 7\}$.
	
	Without loss of generality, then, we may assume $G$ contains a unique minimal normal subgroup $V$, which, as $G$ is solvable, is an elementary abelian $p$-group for some prime $p$. By the induction hypothesis, $\pi(G/V) \subset \{2, 3, 5, 7\}$, meaning we can assume $p \neq 2, 3, 5$ or $7$. This means $V$ is a normal Sylow $p$-subgroup of $G$ and we can consider a $p$-complement $H$ in $G$. Then, $H$ acts on $V$ by conjugation. Furthermore, as $V$ is abelian, an $H$-invariant subgroup $W$ of $V$ would be normal in $G$; by the minimality of $V$, this forces $W = 1$ or $W = V$. Hence, the action is irreducible. 
	
	In both cases, (a) and (b), the Galois group $\Gal(\Q_n{:}\Q)$, with $n = \exp(G)$, acts cyclically on the irrational classes, meaning we can pick representatives for them which are coprime powers of a single element $x$. This means, in case (a), the two irrational classes are of the form $x^G, (x^a)^G$ for $a$ coprime with $o(x)$ and $x^{a^2}$ conjugate to $x$ and, in case (b), they are of the form $x^G, (x^b)^G, (x^{b^2})^G$, with $b$ coprime with $o(x)$ and $x^{b^3}$ conjugate to $x$. So, independently of the case we are considering, we separate the situations $x \in V$ and $x \not \in V$.
	
	\begin{itemize}
		\item[] \underline{Situation 1: $x \not \in V$}
		
		As all the irrational elements of $G$ are conjugate to coprime powers of $x$ and $V$ is normal, $V$ only contains rational conjugacy classes of $G$. Write $\mathbb{F}_p^\times = \langle \mu \rangle$. Then, as $\mu$ is coprime with $p$, we have that, given $v \in V$, there exists some $g \in G$ such that $g^{-1}vg = v^\mu$, by rationality. Using that $V$ is abelian and that $G = HV$, we can take the previous $g$ in $H$. Switching $V$ to additive notation, we can translate the previous observations as follows: given $v \in V$, there exists $h \in H$ such that $h \cdot v = \mu v$. This means the action of $H$ on $V$ has the $(p-1)$-eigenvalue property.
		
		Let $\varphi$ be the Brauer character of the $\mathbb{F}_pH$-module $V$, which, as previously mentioned, is a regular character of $H$, and let $K = \Q(\varphi)$. If $\varphi$ is rational, then \Cref{FariasESoares} finishes the proof in either case. We may thus assume that $\varphi$ is not rational. Also, notice that $H \cong G/V$, meaning $H$ has at most two (resp. three) irrational conjugacy classes and the induction hypothesis applies to $H$. We will finish each case off separately.
		
		\begin{enumerate}[label=(\alph*)]
			\item Suppose we have exactly two irrational classes. Then, $[K: \Q] \leq 2$. By \cite[Theorem A, (c)]{FariasESoares:Primes}, this implies $p = 7$, which contradicts our hypothesis on $V$.
			
			\item Now suppose we have three irrational classes. As in the last case, we have $[K : \Q] \leq 6$ by elementary Galois theory. Also, the action of $H$ on $V$ is irreducible over $\mathbb{F}_p$. If it is reducible over the algebraic closure of $\mathbb{F}_p$, then $\varphi$ is a combination of Galois conjugates of one of its constituents, all with multiplicity $1$, by \cite[Theorem 9.21]{Isaacs:CTFG}. Let $\chi$ be one such (irrational) irreducible constituent. Then, $K = \Q(\varphi) \subset \Q(\chi)$.
			
			Let $\mathcal{G} = \Gal(\Q(\chi) {:} \Q)$. Then each element $\sigma \in \mathcal{G}$ induces a new character $\sigma \chi$. If $\sigma \chi = \chi$, then the fixed field of $\sigma$ contains every value taken by $\chi$, and, therefore, contains $\Q(\chi)$, meaning $\sigma$ is the identity. This means $\sigma \chi \neq \chi$ for all non-identity $\sigma$. 
			
			Since, by \Cref{TheoremA}, there are exactly three irrational characters in $\Irr(G)$, $[\Q(\chi) : \Q] \leq 3$, meaning the same is true for $K$. Also, $[K: \Q]$ cannot be $2$ (again, by elementary Galois theory), meaning $[K : \Q] = 3$, since we assumed $\varphi$ to be irrational. Note, then, that $K$ cannot contain any root of unity besides $1, -1$, as the Euler totient function $\phi$ only takes even values.
			
			By \Cref{FariasESoares}, $K$ contains a primitive $\frac{p-1}{(p-1, 4)}$-root of unity, meaning we have the bound $\frac{p-1}{(p-1, 4)} \leq 2$. The only primes which satisfy this inequality are $p = 2, 3, 5$, all of which contradict our assumption on the order of $V$.
		\end{enumerate}
		
		\item[] \underline{Situation 2: $x \in V$}
		
		\begin{enumerate}[label=(\alph*)]
			\item In this case, where there are exactly two irrational classes, write again $\mathbb{F}_p^\times = \langle \mu \rangle$ and write $\mu_1 = \mu^2$. Writing $B_G(v) = N_G(\langle v \rangle)/C_G(v)$, since there are only two irrational conjugacy classes in $G$, we have $[\Aut(\langle v \rangle) : B_G(v)] \leq 2$ for all $v \in V$. This is because, by \cite[Theorem 3.11]{Navarro:McKay}, $|B_G(v)| = |\Gal(\Q_{o(v)}{:}\Q(v))|$ and, since $\Aut(\langle v \rangle) \cong \Gal(\Q_{o(v)}{:}\Q)$, $[\Aut(\langle v \rangle) : B_G(v)] = |\Gal(\Q(v){:}\Q)| \leq 2$. Also, $\Aut(\langle v \rangle)$ can be identified with $\mathbb{F}_p^\times$, as $v$ has order $p$. Thus, it has a unique subgroup of index $2$, which can be identified with $\langle \mu_1 \rangle$. If $v \neq x$, then $v$ is conjugate to $\mu_1 v$ by hypothesis. Otherwise, $[\Aut(\langle x \rangle) : B_G(x)] = 2$ and $B_G(x) = \langle \mu_1 \rangle$, meaning $x$ too is conjugate to $\mu_1 x$.
			
			The previous paragraph thus leads us to conclude, using that $V$ is abelian, that the action of $H$ on $V$ by conjugation contains the $\frac{p-1}{2}$-eigenvalue property. Then, by \Cref{FariasESoares}, denoting again $K = \Q(\varphi)$ for the Brauer character $\varphi$ of the $\mathbb{F}_pH$-module, $K$ contains a primitive $\frac{(p-1)/2}{((p-1)/2, 4)}$-root of unity. But, in our situation, $H$ is a rational group. By \cite[Theorem A, (c)]{FariasESoares:Primes}, $p \leq 7$ and, if $p = 7$, $K$ would contain a primitive $3$-root of unity, by the above, a contradiction. We are left with $p \in \{2, 3, 5\}$, contradicting our hypothesis on the order of $V$.
			
			\item Finally, suppose once more that there are three irrational classes. We may repeat the exact same argument as last case to finally obtain that the action of $H$ on $V$ has the $\frac{p-1}{3}$-eigenvalue property, with $H$ rational. Then, by \Cref{FariasESoares}, we can get the crude bound of $p \leq 25$. Of these primes, only $7$ and $13$ are possible, since they are the only ones congruent to $1$ modulo $3$. Also, in the case where $p = 13$, $K$ would contain a primitive $3$-root of unity, a contradiction.
		\end{enumerate}			
	\end{itemize}
	
	Thus, we have $p \in \{2, 3, 5, 7\}$, which finishes the proof.
\end{proof}

We note that it is not clear, in the case where $G$ is a solvable group with exactly $2$ rational conjugacy classes, that $7$ can actually divide the order of $G$, as we found no instance of this occurring.

We now proceed to a proof of \Cref{TheoremC}.

\begin{theorem}
	\label{boundIrrational}
	Let $G$ be a finite group. Then:
	\begin{enumerate}[label=(\alph*)]
		\item if $G$ has exactly $n$ irrational conjugacy classes, then $G$ has at most $\frac{n^2}{2}$ irrational irreducible characters;
		\item if $G$ has exactly $n$ irrational irreducible characters, then $G$ has at most $\frac{n^2}{2}$ irrational conjugacy classes;
	\end{enumerate}
\end{theorem}

\begin{proof}
	We will contend ourselves with the proof of case (a), as case (b) follows from similar arguments. Consider the action of $\mathcal{G} = \Gal(\Q_e {:} \Q)$ on the irrational classes of $G$, where $e = \exp(G)$. This gives us an abelian subgroup $H$ of $S_n$ as the image of the action, and thus we may view $H$ as the one acting on the classes. 
	
	Let $C_k$ be a cyclic direct factor of $H$. Then, $C_k$ can act non-trivially on at most $n$ irreducible characters, since it moves at most $n$ conjugacy classes and the two actions are permutation isomorphic by \Cref{BLCT}. Let $m$ be the number of characters moved by $C_k$ and let $C_l$ be another direct factor of $H$. Suppose $C_l$ moves more than $n - m$ irreducible characters all distinct from those moved by $C_k$. Then, taking generators $\sigma \in C_k$ and $\tau \in C_l$, $\sigma \tau$ acts non-trivially on more than $n$ irreducible characters, which contradicts \Cref{BLCT}.
	
	The above argument shows, in particular, that the direct product $C_k \times C_l$ moves at most $2n$ irreducible characters, taking a crude upper bound. Hence, by \Cref{keyCorollary}, we may apply the argument inductively to show that $H$ can act non-trivially on at most $\frac{n^2}{2}$ irreducible characters of $G$. As every irrational character must be moved by some element of $H$, this finishes the proof.
\end{proof}

\begin{proof}[Proof of \Cref{TheoremC}]
	The first part of the theorem is \Cref{boundIrrational}. For the second part, given $\chi \in \Irr(G)$, notice that $[\Q(\chi) : \Q]$ is the size of the orbit of $\chi$ by the action of the group $\Gal(\Q(\chi) {:} \Q)$, since the stabilizer is trivial (this is by an argument done in the proof of \Cref{TheoremB}, for example). Thus, it is bounded by the number of irrational irreducible characters of $G$, which, by \Cref{boundIrrational}, is bounded by a function of $n$.
\end{proof}

In the solvable case, this result also has another interesting consequence.

\begin{corollary}
	\label{corPrimes}
	If $G$ is a solvable group with exactly $n$ irrational irreducible characters, then the prime divisors of $G$ are bounded as a function of $n$.
\end{corollary}

\begin{proof}
	Follows from \Cref{TheoremC} and \cite[Theorem C]{Tent:Quadratic}.
\end{proof}

We note once more that, since \Cref{boundIrrational} is valid both for characters and for classes, the preceding two results also apply if one replaces ``irrational irreducible characters'' for ``irrational conjugacy classes'' and vice-versa.

Having seen how solvable groups with few irrational conjugacy classes resemble rational solvable groups in their prime spectra, one natural question that arises is the following:

\begin{question}
	Let $G$ be a finite group with exactly $n$ irrational conjugacy classes. What are the possible composition factors of $G$?
\end{question}

For nonabelian composition factors, \Cref{TheoremC}, combined with \cite[Theorem A]{FeitSeitz:Rational}, shows the list is relatively simple - it remains to determine which elements in that list may actually occur given a specific number of irrational conjugacy classes. In particular, one very natural question is if the list is finite if $n > 0$, which would entail bounding the possible degrees of alternating composition factors of $G$.

For abelian composition factors, there does not seem to be much known. In \cite{Thompson:Rational}, it is shown that this list is finite for rational groups (the case $n = 0$), but further results have yet to appear, as far as we are aware. In \cite{Moreto:cutGroups}, the author conjectures that the list is always finite, being bounded by a function of the degree of the extensions $\Q(x^G)$ over $\Q$ (and, in particular, by a function of the number of irrational conjugacy classes, by \Cref{TheoremC}).

\section*{Acknowledgments}

The author thanks the Spanish Ministry of Science and Innovation for the financial support, and he would also like to express his gratitude to his PhD advisors Alexander Moretó and Noelia Rizo, as well as to his friend and colleague Juan Martínez Madrid for many helpful conversations on the subject.


\end{document}